\documentclass[12pt]{amsart}
\usepackage[colorlinks=true,citecolor=green,linkcolor=magenta]{hyperref}
\usepackage{amsmath}
\usepackage{verbatim}
\usepackage{amsfonts}
\usepackage{amssymb}
\usepackage{blkarray}
\usepackage{color,colortbl}
\usepackage[all]{xy}  
\usepackage{enumerate}
\usepackage[top=1in, bottom=1in, left=1in, right=1in]{geometry}
\usepackage{mathrsfs}

\usepackage{stmaryrd}
\renewcommand{\k}{\mathbb{K}}
\usepackage{cleveref}
\usepackage{soul}
\usepackage{arydshln,xcolor}
\usepackage{tikz-cd}
\usepackage{graphicx}
\usepackage{comment}
\theoremstyle{definition}
\newtheorem{theorem}{Theorem}[section]
\newtheorem{theoremx}{Theorem}
\usepackage{comment}
\numberwithin{equation}{section}

\newtheorem{corollary}[theorem]{Corollary}
\newtheorem{lemma}[theorem]{Lemma}

\theoremstyle{definition}

\newtheorem{setting}[theorem]{Setting}
\newtheorem{example}[theorem]{Example}

\newtheorem{conjecture}[theorem]{Conjecture}
\newtheorem{questionx}[theoremx]{Question}
\newtheorem{remark}[theorem]{Remark}

\newtheoremstyle{TheoremNum}
{8pt}{8pt}              
{\upshape}                      
{}                              
{\bfseries}                     
{.}                             
{.5em}                             
{\theoremname{#1}\theoremnote{ \bfseries #3}}
\theoremstyle{TheoremNum}

\newcommand{\m}{\mathfrak{m}}

\newcommand{\NN}{\mathbb{N}}

\newcommand{\RR}{\mathcal{R}}

\newcommand{\Pf}{\operatorname{Pf}}

\newcommand{\N}{\mathbb{N}}

\newcommand{\Ass}{\operatorname{Ass}}

\newcommand{\Char}{\operatorname{char}}

\newcommand{\ceil}[1]{\lceil {#1} \rceil}
\newcommand{\floor}[1]{\lfloor {#1} \rfloor}

\renewcommand{\leq}{\leqslant}
\renewcommand{\geq}{\geqslant}

\newcommand{\p}{\mathfrak{p}}

\renewcommand{\b}{\mathfrak{b}}

\renewcommand{\a}{\mathfrak{a}}

\makeatletter
\def\namedlabel#1#2{\begingroup
    #2%
    \def\@currentlabel{#2}%
    \phantomsection\label{#1}\endgroup
}
\makeatother

\address{Department of Mathematical Sciences, New Mexico State University, Las Cruces, New Mexico, U.S.}

\author[Kumar]{Arvind Kumar}
\email{arvkumar@nmsu.edu}
\thanks{The first author, AK, is supported by the AMS Simons Travel Grant 2024.}
\address{Department of Mathematics, Indian Institute of Technology Delhi, India.}

\author[Mukundan]{Vivek Mukundan}
\email{vmukunda@iitd.ac.in}
\thanks{The second author VM is funded by SERB, India under the grant MTR/2023/000308}


\title{Symbolic Powers of Classical varieties}

\begin{document}

\begin{abstract}
     Let $R=\k[x_1,\dots,x_n]$ and let $\a_1,\dots,\a_m$ be homogeneous ideals satisfying certain properties, which include a description of the Noetherian symbolic Rees algebra. We give a solution to a question of  Harbourne and Huneke for this set of ideals. We also compute the Waldschmidt constant and resurgence and show that it exhibits a stronger version of the Chudnovsky and Demailly-type bounds. We further show that these properties are satisfied for classical varieties such as the generic determinantal ideals, minors of generic symmetric matrices, generic extended Hankel matrices, and ideal of pfaffians of skew-symmetric matrices. 
\end{abstract}
\maketitle
    \section{Introduction}
     The main objective of this article is to explore questions regarding the symbolic powers of classical varieties. Let $R$ be a Noetherian ring and $\a$ be a non-zero proper ideal. The fundamental object we explore is the {\it $n^{th}$ symbolic power} of the ideal $\a$, defined as $\a^{(n)}=\cap_{\p\in\Ass{\a}}(\a^nR_\p\cap R)$. Generally, it can be a challenge to compute symbolic powers. 
     In this context, a natural question arises regarding the proximity between the ordinary and symbolic powers of ideals, particularly whether their topologies are equivalent or cofinal. Thus, it is desirable to obtain sharper bounds on $m>n$ such that $\a^{(m)}\subseteq \a^n$, referred to as the \textit{containment problem}. 
     The topologies being cofinal was proved by Schenzel and it was further refined by Swanson \cite{swanson00}. She provided a landmark contribution by showing the existence of a constant $c$ such that $\a^{(cn)}\subseteq \a^n$ for all $n\geq 1$. Subsequent advancements in the study of regular rings revealed that this constant $c$ coincides with a key invariant of the ideal $\a$, namely its \textit{big height}, see \cite{ELS01, HH02, MS17}. The big height measures the largest height among all the associated primes of $\a$. 
 Intriguingly, it was observed that a smaller symbolic power can be contained in the $n^{th}$ ordinary power. For example, some classes of ideals satisfy $\a^{(hn-h+1)}\subseteq \a^n$ for all $n\ge 1$ \cite{HunekeHarbourne13}. The latter containment is called the \textit{Harbourne conjecture}, see \cite{seshariconstantcounterexample}, and counterexamples to this question do exist. For further developments in this direction, please refer to the following references \cite{seshariconstantcounterexample, DD20, GH19, GHM21, GHM20, G20}. A stronger version of this conjecture is also sought after in \cite{HunekeHarbourne13}. The latter reference has a lot of questions regarding the containment problem and also posits that certain symbolic powers are lying sufficiently deep inside the ordinary powers with coefficients in the maximal ideals. 
     \begin{questionx}[Harbourne-Huneke, \cite{HunekeHarbourne13}] \label{ques-A}
         Let $I\subseteq R$ be a homogeneous ideal. For which $m,i,j$ do we have $I^{(m)}\subseteq \m^jI^i$?
     \end{questionx}
     The above question is in fact a generalized version of the Eisenbud-Mazur conjecture \cite{Eisenbud-Mazur97}, which states that $I^{(2)}\subseteq \m I$ for a prime ideal $I$ in $R$. One of the questions we explore in this article is the solution to \Cref{ques-A} in the context of classical varieties. Bocci and Harbourne \cite{BH10} introduced the notion of \textit{resurgence} to facilitate the study of the containment problem. The resurgence of the ideal $\a$ is defined as $\rho(\a)=\sup\left\{\frac{s}{r}~|~\a^{(s)}\not\subseteq \a^r \right\}$, and it measures the extent to which the ordinary and symbolic power containment deviates from the ideal scenario.  It offers a quantitative measure of this deviation, aiding researchers in their exploration of the containment problem. Notably, it is apparent that if $m/n>\rho(\a)$, then $\a^{(m)}\subseteq \a^n$.   An asymptotic version of resurgence was introduced by Guardo, Harbourne, and Van Tuyl in \cite{EBA}. 
The \textit{asymptotic resurgence} of an ideal $\a$ is defined as $\widehat\rho(\a)=\sup\left\{ \frac{s}{r}~|~ \a^{(st)} \not\subset \a^{rt}\text{ for  }t\gg0\right\}$.
It is also known from \cite{EBA} that  $\displaystyle 1\leq \frac{\alpha(\a)}{\hat{\alpha}(\a)}\leq\widehat\rho(\a)\leq \rho(\a),$ where  $\alpha(\a)$ denotes the minimal degree of an element in $\a$ and $\widehat{\alpha}(\a)$ denotes the {\it Waldschmidt constant} defined as $\displaystyle \widehat{\alpha}(\a) = \underset{s \to \infty}{\lim}\frac{\alpha(\a^{(s)})}{s}$. This inequality highlights the connections between various invariants. 
A natural question arises: under what conditions does equality $\frac{\alpha(\a)}{\widehat{\alpha}(\a)}=\widehat{\rho}(\a)=\rho(\a)$ hold? As observed earlier, the symbolic powers $\a^{(r)}$ are computationally expensive to determine. Consequently, the $\alpha(\a^{(r)})$ becomes increasingly difficult to compute as well. Therefore, any information regarding $\alpha(\a^{(r)})$ invariably leads to a deeper understanding of the symbolic powers themselves. 

The questions we answer in this article are the following. 
\begin{questionx} Let $\a$ be the ideal defining a classical variety in a polynomial ring $R$ and its homogeneous maximal ideal $\m$. 
    \begin{enumerate}
        \item Does equality  $\frac{\alpha(\a)}{\widehat{\alpha}(\a)}=\widehat{\rho}(\a)=\rho(\a)$ hold?
        \item Fix $r$. What is the value of the invariant $\alpha(\a^{(r)})$?
        \item Fix $N$. For what values of $s,r$, does $\a^{(s)} \subseteq \m^N\a^r$?
    \end{enumerate}
\end{questionx}

In an effort to study Question B(3) for classical varieties, we use the invariants resurgence and asymptotic resurgence associated with a pair of graded families of ideals, as introduced in \cite{HKNN23}. Let $\a_\bullet=\{\a_i\}_{i \ge 1}$ and $\b_\bullet=\{\b_i\}_{i \ge 1}$ be two graded families of ideals, then the \textit{resurgence} and \textit{asymptotic resurgence} of the ordered pairs $(\a_\bullet,\b_\bullet)$ are
\begin{align*}
    \rho(\a_\bullet,\b_\bullet)=\sup\left\{\frac{s}{r}~|~s,r\in\mathbb{N},\a_s\not\subseteq \b_r\right\}, &&
      \widehat{\rho}(\a_\bullet,\b_\bullet)=\left\{\frac{s}{r}~|~s,r\in\mathbb{N},\a_{st}\not\subseteq\b_{rt}\text{ for } t\gg 1\right\}.
\end{align*}
If $\a_\bullet=\{\a^{(i)}\}_{i\geq 1},\b_\bullet=\{ \a^i\}_{i\geq 1}$, then the notion of resurgence for the pair coincide with the usual notion of the resurgence. If we set $\a_\bullet=\{\a^{(i)}\}_{i \ge 1}$ and $\b_\bullet=\{\m^N\a^i\}_{i \ge 1}$ with fixed $N$, then the resurgence of the pairs would answer Question B(3) asymptotically.

    
The classical varieties we study in this article include generic determinantal ideals, minors of generic symmetric matrices, generic extended Hankel matrices, and ideals of pfaffians of skew-symmetric matrices. All these varieties exhibit a Noetherian symbolic Rees algebra (see Section \ref{sec 2} for a definition of the symbolic Rees algebra). Furthermore, these classical varieties exhibit a unique feature: a recursive description of their symbolic Rees algebras.  
 
 This structure, presented below, offers valuable insights and facilitates deeper exploration of their properties. For our purposes, we set $R=\k[x_1,\dots,x_n]$ and $\a_1,\dots,\a_m$ be a sequence of non-zero proper homogeneous ideals in $R$. The results in this article depend on whether the sequence of ideals satisfies all or some of the following properties.
    \begin{enumerate}
        \item[\namedlabel{P1}{(P1)}] for all $1\leq t\leq m$, the symbolic Rees algebra $\mathcal{R}_s(\a_t)=R[\a_tT,\a_{t+1}T^2,\dots,\a_mT^{m-t+1}]$,
        \item[\namedlabel{P2}{(P2)}] for all $0\leq t\leq m-1$, $\alpha(\a_{t+1})=\alpha(\a_1)+t$.
    \end{enumerate}
Some of the main results in this article have been collected in the following theorem.
\begin{theoremx}[\Cref{wald-family,chudnovsky and demailly bounds}]
    Let $R=\k[x_1,\dots,x_n]$ and $\a_1,\dots,\a_m$ be homogeneous ideals satisfying the properties \ref{P1} and \ref{P2} mentioned above, then
    \begin{enumerate}
        \item For all $k \in \NN$ and  $1\leq l\leq m-t+1$, \begin{align*}
        \alpha\left(\a_t^{(k(m-t+1)+l)}\right)= 
                k(\alpha(\a_1)+m-1)+\alpha(\a_1)+t-2+l .
        \end{align*}
Furthermore, the Waldschmidt constant is given by the equation $$\widehat{\alpha}(\a_t)=\frac{\alpha(\a_1)+m-1}{m-t+1}.$$
\item Let $1 \le t <m$. Then, the following are equivalent 
\begin{enumerate}
\item for all  $s,r\ge 1$,
    \begin{align*}
        \frac{\alpha\left(\a_t^{(s)}\right)}{s}\geq\frac{\alpha\left(\a_t^{(r)}\right)+N-1}{r+N-1}.
    \end{align*}
    \item  for all  $s\ge 1$,
    \begin{align*}
        \frac{\alpha\left(\a_t^{(s)}\right)}{s}\geq\frac{\alpha\left(\a_t\right)+N-1}{N}.
    \end{align*}
    \item $N \ge m-t+1.$
\end{enumerate}
\end{enumerate}
\end{theoremx}

The ideals $\a_t$ take the role of the $t\times t$  minors (or $2t\times 2t$ pfaffians) defining the classical varieties discussed above. A stronger version of the Chudnovsky and Demailly-type bounds (see \Cref{sec 2} for the definition of these bounds) is expected to hold for these classical varieties.  For all the classical varieties we study in this article, we obtain the Waldschmidt constant and better versions of the Chudnovsky and Demailly bounds through the above theorem. We prove sharper inequalities for both types of bounds compared to previously established results, such as those found in \cite{BGHN22} for generic determinantal ideals, minors of generic symmetric matrices, and pfaffians of skew-symmetric matrices. The bounds established in \cite{BGHN22} depend on the big height, and our version of the bounds uses constants much smaller than the big height, thus, making our bounds closer to the version given by their respective conjectures (Chudnovsky and Demailly). The Demailly bound is much stronger than the Chudnovsky bound. However, \Cref{chudnovsky and demailly bounds} also demonstrates that the Chudnovsky and Demailly bounds are equivalent for the sequence of ideals that satisfy \ref{P1} and \ref{P2}. This crucial feature is emphasized in \Cref{chudnovsky and demailly bounds}, and it is not expressed in any of the literature known to the authors.

We now turn our attention to Question B(3).  In addition to properties \ref{P1} and \ref{P2}, assume that the ideals $\a_1,\dots,\a_m$ satisfy one of the following property for a fixed non-negative integer $N$.
\begin{enumerate}    
        \item[\namedlabel{P3}{(P3)}] $\a_1$ is generated by linear forms and  for every $1 \le t \le m$, and $s \ge 1$ $$\overline{(\m^N\a_t)^s}= \a_1^{Ns+ts} \bigcap \left( \cap_{j=2}^t \a_{j}^{(s(t-j+1))}\right).$$
        \item[\namedlabel{P4}{(P4)}] $\a_1$ is generated by linear forms and  for every $1 \le t \le m$, $$\overline{\m^N\a_t^s}= \a_1^{N+ts} \bigcap \left( \cap_{j=2}^t \a_{j}^{(s(t-j+1))}\right).$$
\end{enumerate}
\begin{theoremx}[\Cref{asym-res-HH1,asym-res-HH}]
Let $R=\k[x_1,\dots,x_n]$ and $\a_1,\dots,\a_m$ be homogeneous ideals satisfying the properties \ref{P1} and \ref{P2} mentioned above. 
\begin{enumerate}
    \item If property \ref{P3} holds, then for all $1 \le t \le m$, 
    $$ \widehat{\rho}(\a_t^{(\bullet)},(\m^N\a_t)^{\bullet} ) = \frac{N+\alpha(\a_t)}{\widehat{\alpha}(\a_t)}=\frac{(N+t)(m-t+1)}{m}.$$
    \item If property \ref{P4} holds, then for all $1 \le t \le m$, 
    $$ \widehat{\rho}(\a_t^{(\bullet)},\overline{\m^N(\a_t)^{\bullet} }) = \frac{\alpha(\a_t)}{\widehat{\alpha}(\a_t)}=\frac{t(m-t+1)}{m}.$$
    \end{enumerate}
    where $\widehat{\rho}(\a_\bullet,\b_\bullet)$ is the asymptotic resurgence for a pair of graded families of ideals $\a_\bullet,\b_\bullet$.
\end{theoremx}

All the classical varieties satisfy the properties \ref{P1}--\ref{P4}. As a direct corollary of the above theorem, we obtain the equality $\widehat{\rho}(\a_t) = \frac{\alpha(\a_t)}{\widehat{\alpha}(\a_t)} = \frac{t(m-t+1)}{m}$, thereby providing asymptotic solutions to Question B(1) and Question A for classical varieties.

The authors in \cite{HKNN23} established that $\widehat{\rho}(\a_\bullet,\b_\bullet)$ coincides with the expected value (analogous to the value as in the above theorem) provided the Rees algebra of the second filtration $\mathcal{R}(\b_\bullet)$ is Noetherian. Interestingly, the Rees algebra $\mathcal{R}(\m^N(\a_t)^\bullet)$ is never Noetherian when $N$ is positive. This follows from the fact that if $\mathcal{R}(\mathfrak{m}^N(\mathfrak{a}_t)^\bullet)$ were Noetherian, then for some positive integer $k$, $\mathfrak{m}^N(\mathfrak{a}_t)^{ks}=\left(\mathfrak{m}^N(\mathfrak{a}_t)^{k}\right)^s$ for all $s$. This cannot be true when $N$ is positive.  This renders the above theorem very significant. Furthermore, it provides a positive answer and supporting evidence for Question 2.19 in \cite{HKNN23}. 

The techniques used in this article are more wide-reaching and are not restricted to the case of classical varieties alone (see \Cref{self linked example,star config example}). Any class of ideals satisfying the above conditions becomes a candidate for consideration, paving the way for wide-ranging applications of our findings.

 We give a brief review of each section. \Cref{sec 2} sets up the tools to study the Waldschmidt constant and the Chudnovsky, Demailly-type bounds. These are presented in \Cref{wald-family,chudnovsky and demailly bounds}. In \Cref{Sec3}, we introduce properties \ref{P3} and \ref{P4}, in addition to properties \ref{P1} and \ref{P2}. Using these properties, we show the main results in \Cref{asym-res-HH1,asym-res-HH} and obtain \Cref{asym-res} as a consequence. \Cref{Application sections} presents details of the classical varieties we study in this article. We also individually list the results for each classical variety for easy information access. Finally, we note that properties \ref{P1} and \ref{P2} hold for ladder determinantal ideals, whereas properties \ref{P3} and \ref{P4} do not. This raises the question: do the results on asymptotic resurgence extend to classical ladder determinantal ideals?

\subsection{Acknowledgements} We are grateful to H\`a, Huy T\`ai, and Nguy\^en, Thai Th\`anh, for their valuable suggestions, improvements, and references for the earlier version of the paper.

    \section{Waldschimdt Constant, Demailly and Chudnovsky type bounds}\label{sec 2}
This section focuses on the Waldschimdt Constant, Demailly, and Chudnovsky-type bounds of a finite family of ideals in which each symbolic Rees Algebra is Noetherian and satisfies recursive properties \ref{P1} and \ref{P2}. 

    Let $R=\k[x_1,\ldots,x_n]$ be a standard graded polynomial ring over an arbitrary field $\k.$ Let $\a \subseteq R$ be a non-zero proper homogeneous ideal. The  \emph{symbolic Rees algebra} of $\a$, denoted by $\RR_s(\a)$, is defined to be
	$$\RR_s(\a) = \bigoplus_{n \ge 0} \a^{(n)} T^n \subseteq R[T].$$
 Generally, the symbolic Rees algebra of an ideal is not necessarily Noetherian (cf. \cite{C1991, H1982, R1985}). We introduce a family of ideals in which the symbolic Rees algebra of each ideal is Noetherian and satisfies recursive properties. We study these properties primarily because they are satisfied for all the classical varieties we consider in this article (see \Cref{Application sections}).

In this section and the sequel, we set $R=\k[x_1,\dots,x_n]$ and $\a_1,\dots,\a_m$ be a sequence of non-zero proper homogeneous ideals in $R$ satisfying \ref{P1} and \ref{P2}. 
\begin{remark}\label{rmk t=m}
  Assume $t=m$. Then, $\RR_{{s}}(\a_m) =R[\a_mT]$, which is the Rees algebra of $\a_m$. Therefore, $ \a_m^{(s)}= \a_m^s$ for all $s$, and hence, $\widehat{\alpha}(\a_m) = \alpha(\a_m)=\alpha(\a_1)+m-1$. 
\end{remark}

    Next, we prove a few auxiliary lemmas that play an important role in computing the least degree of each symbolic power of $\a_t$, and hence, the Waldschmidt constant of $\a_t$. 
\begin{lemma}\label{min}
Let $t, m, b$ be positive integers with $t \le m$. For any positive integer $k$, if $a_1,\ldots, a_{m-t+1} $ are non-negative integers with $\sum_{i=1}^{m-t+1} ia_i =k(m-t+1)$, then $$\sum_{i=1}^{m-t+1} (b+i)a_i \ge k(b+m-t+1).$$ Moreover, the inequality is an equality if $a_{m-t+1} =k$ and $a_i=0$ for $1 \le i \le m-t$. 
\end{lemma}
\begin{proof}
First, we note that $\sum_{i=1}^{m-t+1} a_i \ge k.$ Indeed, if $\sum_{i=1}^{m-t+1} a_i \le k-1,$ then $\sum_{i=1}^{m-t+1} ia_i \le (m-t+1) \sum_{i=1}^{m-t+1} a_i \le (m-t+1)(k-1)$ which is a contradiction. Thus, $\sum_{i=1}^{m-t+1} (b+i)a_i = b\sum_{i=1}^{m-t+1} a_i+\sum_{i=1}^{m-t+1} ia_i\ge bk+k(m-t+1)=k(b+m-t+1).$ 

Next, if we take $a_{m-t+1} =k$ and  $a_i=0$ for $1 \le i \le m-t$, then  $\sum_{i=1}^{m-t+1} (b+i)a_i =k(b+m-t+1).$ Hence, the lemma follows. 
\end{proof}

\begin{lemma} \label{min-rem}
Let $t, m, b$ be positive integers with $t < m$ and $l$ be a positive integer such that $l\leq m-t$. For every $k\in\mathbb{N}$, if $a_1,\dots,a_{m-t+1}$ are  non negative integers such that $\sum_{i=1}^{m-t+1} ia_i=k(m-t+1)+l$, then 
    \begin{align*}
        \sum_{i=1}^{m-t+1}(b+i)a_i\geq k(b+m-t+1)+b+l.
    \end{align*}
Moreover, inequality is an equality if $a_l=1,a_{m-t+1}=k$ and $a_i=0$ for $i\not\in \{l,m-t+1\}$.
\end{lemma}
\begin{proof}
First, we note that $\sum_{i=1}^{m-t+1} a_i \ge k+1.$ Indeed, if $\sum_{i=1}^{m-t+1} a_i \le k,$ then $\sum_{i=1}^{m-t+1} ia_i \le (m-t+1) \sum_{i=1}^{m-t+1} a_i \le k(m-t+1)$ which is a contradiction. Thus, $\sum_{i=1}^{m-t+1} (b+i)a_i = b\sum_{i=1}^{m-t+1} a_i+\sum_{i=1}^{m-t+1} ia_i\ge b(k+1)+k(m-t+1)+l=k(b+m-t+1)+b+l.$ 

Next, if we take $a_{m-t+1} =k$, $a_l=1$  and  $a_i=0$ for $i \not\in \{l, m-t+1\}$, then  $\sum_{i=1}^{m-t+1} (b+i)a_i =k(b+m-t+1)+b+l.$ Hence, the lemma follows. 
\end{proof}

Now, we compute the least degree of each symbolic power of $\a_t$, and the Waldschmidt constant of $\a_t.$
\begin{theorem}\label{wald-family}
Assume $R=\k[x_1,\dots,x_n]$ and $\a_1,\dots,\a_m$ be a sequence of non-zero proper homogeneous ideals in $R$ satisfying \ref{P1} and \ref{P2}. 
Then, for $1 \le t < m$,
\begin{align*}
        \alpha\left(\a_t^{(k(m-t+1)+l)}\right)=
                k(\alpha(\a_1)+m-1)+\alpha(\a_1)+t-2+l,
       \end{align*}  for all $ k \in \NN$  and $1\leq l\leq m-t+1.$
Furthermore, the Waldschmidt constant is given by the equation $$\widehat{\alpha}(\a_t)=\frac{\alpha(\a_1)+m-1}{m-t+1}.$$
\end{theorem}

\begin{proof}
Let $1 \le t < m$.    Since $\mathcal{R}_s(\a_t)=R[\a_tT,\a_{t+1}T^2,\ldots,\a_mT^{m-t+1}]$, for any $s\geq 2$,
    \begin{align*}
        \a_t^{(s)}=\sum_{\substack{a_i\in \mathbb{N},\\\sum_{i=1}^{m-t+1}ia_i=s}}\a_t^{a_1}\cdots \a_m^{a_{m-t+1}}.
    \end{align*}
    Therefore,
    \begin{align*}
        \alpha\left(\a_t^{(s)}\right)&=\min\left\{ \sum_{i=1}^{m-t+1}a_i\alpha(\a_{t-1+i})~|~a_1,\dots,a_{m-t+1}\in \NN\text{ and } \sum_{i=1}^{m-t+1}ia_i=s\right\}\\
        &=\min\left\{ \sum_{i=1}^{m-t+1}a_i\left(\alpha(\a_1)+t-2+i\right)~|~a_1,\dots,a_{m-t+1}\in \NN \text{ and } \sum_{i=1}^{m-t+1}ia_i=s \right\}
    \end{align*}
    Write $s=k(m-t+1)+l$ with $k\in \NN$ and $1\leq l\leq m-t+1$. Then from \Cref{min,min-rem}, we have 
    \begin{align*}
        \alpha\left(\a_t^{(s)}\right)=\begin{cases}
                (k+1)(\alpha(\a_1)+m-1) & \text{ if } l=m-t+1\\
                k(\alpha(\a_1)+m-1)+\alpha(\a_1)+t-2+l&\text{ if }1\leq l\leq m-t.
        \end{cases}
    \end{align*}
    Consider 
    \begin{align*}
        \widehat{\alpha}(\a_t)=\lim_{s\rightarrow \infty}\frac{\alpha\left(\a_t^{(s)}\right)}{s}=\lim_{k\rightarrow \infty}\frac{\alpha\left(\a_t^{(k(m-t+1))}\right)}{k(m-t+1)}=\lim_{k\rightarrow \infty}\frac{k(\alpha(\a_1)+m-1)}{k(m-t+1)}=\frac{\alpha(\a_1)+m-1}{m-t+1}.
    \end{align*}  This completes the second assertion. One could also use \cite[Theorem 3.6]{DG20} along with the first assertion to obtain the formula for $\widehat{\alpha}(\a_t)$.
\end{proof}

Though the above theorem proves instrumental in studying classical varieties in Section~\ref{Application sections}, its applicability extends well beyond this context. The following example demonstrates its usefulness in a non-classical variety setting.
\begin{example}\label{self linked example}
  Let $R=\k[x,y,z]$ be with standard grading and $\m=(x,y,z)$. Let $I$ to be the ideal generated by $2 \times 2$ minors of the matrix 
  $$ A= \begin{bmatrix}
      x & y & a_1x+b_1y+c_1z\\
      y & z& a_2x+b_2y+c_2z
  \end{bmatrix},$$ where $a_i,b_i,c_i \in \k.$ If $\mu(I) \le 2$, then $I$ is a complete intersection ideal and therefore, $I^{(s) } =I^s$ for all $s.$ So, we assume that  $\mu(I) =3.$ In this case, $I$ is height two ideal in $R$.   For $1 \le i \le 3$, set $f_i $  to be the $2 \times 2$ minors of the matrix $A$ obtained by deleting the $i^{th}$ column. Note that  $I=(f_1,f_2,f_3)$, and therefore, $\alpha(I)=2. $ Next, consider the following matrix 
  $$ \Gamma = \begin{bmatrix}
      f_1+a_1f_3 & -f_2+b_1f_3 & c_1f_3\\
      a_2f_3 & f_1+b_2f_3& -f_2+c_2f_3
  \end{bmatrix}.$$ For $1 \le i \le 3$, set $D_i $  to be the $2 \times 2$ minors of the matrix $\Gamma$ obtained by deleting the $i^{th}$ column.  Then, it follows from \cite[Section 2]{HU90} that  $\RR_s(I) =R[IT,wT^2]$ where $w =\frac{D_1}{x} =\frac{D_2}{y}=\frac{D_3}{z}.$ Notice that $\deg(D_i) =4$, and therefore, $\deg (w) =3.$ Thus, the  sequence $\a_1=\m, \a_2 =I$ and $\a_3=(w)$ satisfy properties \ref{P1} and \ref{P2}. Hence, by \Cref{wald-family}, $\widehat{\alpha}(I) = \frac{3}{2}. $ 

  Next, we show that $\rho(I)=\widehat{\rho}(I) =\frac{\alpha(I) }{\widehat{\alpha}(I)}=\frac{4}{3}.$ First, note that $I^{(2)} =I^2+(w) $. Therefore, $\m I^{(2)} = \m I^2 + (D_1,D_2,D_3) \subseteq I^2 $. Since $\deg(w) =3$ and $\alpha(I) =2$ we get that  $I^{(2)} \subseteq \m I. $ It follows from \cite[Theorem 6.2]{GHM20} that $ \rho(I) \le \frac{4}{3}. $ The rest of the assertion follows from the fact that $\displaystyle \frac{\alpha(\a)}{\hat{\alpha}(\a)}\leq\widehat\rho(\a)\leq \rho(\a).$ 
 \end{example}
We now give bounds for $\frac{\alpha\left(\a_t^{(s)}\right)}{s}$, one of the main theorems in this section. The relevance of this theorem is explained in the sequel.
\begin{theorem}\label{chudnovsky and demailly bounds}
Assume $R=\k[x_1,\dots,x_n]$ and $\a_1,\dots,\a_m$ be a sequence of non-zero proper homogeneous ideals in $R$ satisfying \ref{P1} and \ref{P2}.   Let $1 \le t <m$. Then, for $N \in \NN$,  the following are equivalent 
\begin{enumerate}
\item for all  $s,r\ge 1$,
    \begin{align*}
        \frac{\alpha\left(\a_t^{(s)}\right)}{s}\geq\frac{\alpha\left(\a_t^{(r)}\right)+N-1}{r+N-1}
    \end{align*}
    \item  for all  $s\ge 1$,
    \begin{align*}
        \frac{\alpha\left(\a_t^{(s)}\right)}{s}\geq\frac{\alpha\left(\a_t\right)+N-1}{N}
    \end{align*}
    \item $N \ge m-t+1.$
\end{enumerate}
\end{theorem}
\begin{proof}
The result is vacuously true if $t=1$ and $\alpha(\a_1)=1$. So, we assume that $t+\alpha(\a_1) >2$.   Notice that  $(1)$ implies $(2)$ follows immediately by choosing $r=1$. Assume $(2) $  holds.   Using \Cref{wald-family}, we get 
    \begin{align*}
        \widehat{\alpha}(\a_t)=\frac{\alpha(\a_1)+m-1}{m-t+1}\geq\frac{\alpha\left(\a_t\right)+N-1}{N}.
    \end{align*} 
Therefore $N (\alpha(\a_1)+m-1) \ge (m-t+1)(\alpha(\a_1)+N+t-2)$ which implies that $N (\alpha(\a_1)+t-2) \ge (m-t+1)(\alpha(\a_1)+t-2)$. We conclude that  $N \ge m-t+1$. Thus, if $(2)$ holds, then $(3)$ holds as well. 

Next, we prove $(3)$ implies $(1). $
It is sufficient to prove that for all  $r\in\mathbb{N}$,
    \begin{align}\label{demailly eq1}
        \widehat{\alpha}(\a_t)=\frac{\alpha(\a_1)+m-1}{m-t+1}\geq\frac{\alpha\left(\a_t^{(r)}\right)+N-1}{r+N-1}. 
    \end{align}
    Let $r\in \NN$. Write  $r=k(m-t+1)+l$ for some $k\in \NN$ and $1\leq l\leq m-t+1$. By \Cref{wald-family}, we have 
    \begin{align*}
        \alpha\left(\a_t^{(r)}\right)=
                k(\alpha(\a_1)+m-1)+\alpha(\a_1)+t-2+l.
    \end{align*} We split the proof into cases.

First, suppose $l=m-t+1$. Since $m-t+1\leq \alpha(\a_1)+m-1$, we get $(N-1)(m-t+1)\leq (N-1)(\alpha(\a_1)+m-1)$. This implies that $(N-1)(m-t+1)+(k+1)(\alpha(\a_1)+m-1)(m-t+1)\leq (N-1)(\alpha(\a_1)+m-1)+(k+1)(\alpha(\a_1)+m-1)(m-t+1)$.  Equivalently, we get that $(m-t+1)\left(N-1+(k+1)(\alpha(\a_1)+m-1)\right)\leq \left(N-1+(k+1)(m-t+1)\right)(\alpha(\a_1)+m-1)$. 
 Therefore, using \Cref{wald-family} we get
    \begin{align*}
         \frac{\alpha(\a_1)+m-1}{m-t+1}\geq \frac{\alpha\left(\a_t^{(r)}\right)+N-1}{r+N-1} .
    \end{align*}
    Note that this part of the proof is true for all $N\geq 2$.

Next, suppose $1 \le  l \le m-t$.   Then $N-1\geq m-t+1 -l$ and this is equivalent to $(\alpha(\a_1)+t-2)(l+N-1)\geq (\alpha(\a_1)+t-2)(m-t+1)$. After rearranging, we see that $(\alpha(\a_1)+m-1)(l+N-1)\geq (\alpha(\a_1)+t+l+N-3)(m-t+1)$ which is further equivalent to
    \begin{multline*}
        (\alpha(\a_1)+m-1)k(m-t+1)+(\alpha(\a_1)+m-1)(l+N-1)\geq\\ (\alpha(\a_1)+m-1)k(m-t+1)+(\alpha(\a_1)+t-3+l+N)(m-t+1).
    \end{multline*}
The latter inequality can be simplified as 
    \begin{multline*}
        (\alpha(\a_1)+m-1)(k(m-t+1)+l+N-1)\geq \\ (m-t+1)((\alpha(\a_1)+m-1)k+(\alpha(\a_1)+t-3+l+N)).
    \end{multline*} 
    Therefore, for $r=k(m-t+1)+l$ with $k \in \NN$ and $1\leq l\leq m-t$, using \Cref{wald-family}, we have
    \begin{align*}
        \frac{\alpha(\a_1)+m-1}{m-t+1}\geq \frac{\alpha\left(\a_t^{(r)}\right)+N-1}{r+N-1}.
    \end{align*}
 Hence, the assertion follows.   
\end{proof}
\begin{remark}
Assume $t= m$. Then, it follows from \Cref{rmk t=m} that, $\a_m^s =\a_m^{(s)}$ for all $s$ and $\widehat{\alpha}(\a_m) = \alpha(\a_m)=\alpha(\a_1)+m-1$. Thus,   for all  $s,r, N\in\mathbb{N}$, we have
    \begin{align*}
         \frac{\alpha\left(\a_m^{(s)}\right)}{s}\geq\frac{\alpha\left(\a_m^{(r)}\right)+N-1}{r+N-1}.
    \end{align*}
    \end{remark}

We explain the relevance and importance of the above theorem now. While studying the least degree homogeneous polynomial vanishing at a given set of points $X \subset P^N_{\k}$ up to a certain order, Chudnovsky in  \cite{C1981} gave evidence for and conjectured the following statement.

\begin{conjecture} Assume that $\k$ is algebraically closed field of characteristic $0.$ Let $X \subset P^N_{\k}$ be a set of distinct points and $\a$ be the defining ideal of $X$. Then, for all $s \ge 1$, $$\frac{\alpha\left(\a^{(s)}\right)}{s}\geq\frac{\alpha\left(\a\right)+N-1}{N}.$$
\end{conjecture}
Statement (2) of \Cref{chudnovsky and demailly bounds} establishes a similar bound for the family of ideals satisfying the conditions outlined in our introduction. Notably, we express the constant $N$ in terms of the generating degree of the symbolic Rees algebra of $\a_t$. Furthermore, for classical varieties (see Section~\ref{Application sections}), we show that the integer $N$ can be significantly smaller than both the number of variables in $R$ and the big height of the ideal $\a_t$. It is worth noting that in \Cref{chudnovsky and demailly bounds}, we present an equivalent condition for the applicability of the Chudnovsky-like bound compared to the initial version of the conjecture. This, in essence, demonstrates the sharpness of the established bounds. 

Chudnovsky's conjecture was generalized by Demailly \cite{D1982} in the following:

 \begin{conjecture} Assume that $\k$ is algebraically closed field of characteristic $0.$ Let $X \subset P^N_{\k}$ be a set of distinct points and $\a$ be the defining ideal of $X$. Let $r$ be any positive integer. Then, for all $s \ge 1$, $$\frac{\alpha\left(\a^{(s)}\right)}{s}\geq\frac{\alpha\left(\a^{(r)}\right)+N-1}{r+N-1}.$$
\end{conjecture}

Statement $(1)$ of \Cref{chudnovsky and demailly bounds} gives a Demailly-like bound for the family of ideals that satisfy \ref{P1} and \ref{P2}. Importantly, all the positive attributes previously highlighted for the Chudnovsky-type bounds also apply to the Demailly-type bound.

Of course, the Demailly bound is much stronger than the Chudnovsky bound. On the other hand, \Cref{chudnovsky and demailly bounds} also shows that the Chudnovsky and Demailly bounds are equivalent for the class of ideals that satisfy \ref{P1} and \ref{P2}. This is a crucial feature emphasized in \Cref{chudnovsky and demailly bounds}, not expressed in the literature known to the authors.


\section{Resurgence and asymptotic resurgence numbers}\label{Sec3}
This section is dedicated to the computation of resurgence and asymptotic resurgence numbers for pairs of graded families of ideals. Building upon the definition established in \cite{HKNN23}, we present two primary theorems that provide asymptotic solutions to Questions A and B(3). These theorems are more widely applicable (see \Cref{star config example} and \Cref{Application sections}) for a finite family of ideals that satisfies \ref{P1}, \ref{P2}, and \ref{P3} or \ref{P4}. As a corollary, we derive a formula for the asymptotic resurgence of classical varieties, a topic explored further in \Cref{Application sections}.      

Given an ideal $\a$, we use the notation $\a^{(\bullet)}$ to represent the graded family consisting of symbolic powers of $\a$, which is $\a^{(\bullet)}=\{\a^{(i)}\}_{i\ge 1}$. 
We represent the graded family consisting of ordinary powers of $\a$ by $\a^{\bullet}$, which is  $\a^{\bullet}=\{\a^{i}\}_{i\ge 1}$.
Similarly, the graded family consisting of integral closure of powers of $\a$ is denoted by $\overline{\a^{\bullet}}$, expressed as $\overline{\a^{\bullet}}=\{\overline{\a^{i}}\}_{i\ge 1}$.

 Let $\a_\bullet$ be a graded family of homogeneous ideals in $R$. Then, the sequence $\{\alpha(\a_{n})\}_{n \ge 1}$ is a sub-additive sequence due to the graded property of $\a_\bullet$. Thus, by Fekete's Lemma, the limit $\lim\limits_{n \rightarrow \infty} \frac{\alpha(\a_{n})}{n}$ exists and is equal to $\inf\limits_{n \in \NN} \frac{\alpha(\a_n)}{n}$, which is the \emph{Waldschmidt constant} of $\a_\bullet$ and is denoted by $\widehat{\alpha}(\a_\bullet)$, see \cite{HKNN23} for more details.

Numerous results established for standard resurgence also hold for the resurgence of a pair of graded families of ideals, see \cite{HKNN23}. Consequently, we aim to address Question B in its entirety using an analogous version of $\widehat{\rho}(\a_\bullet,\b_\bullet)$. Now, we present one of the main results of this section. 

\begin{theorem}\label{asym-res-HH1}
  Assume $R=\k[x_1,\dots,x_n]$ and $\a_1,\dots,\a_m$ be a sequence of non-zero proper homogeneous ideals in $R$ satisfying \ref{P1} and \ref{P2}. Let $N$ be a non-negative integer so that \ref{P3} holds. Then, for all $1 \le t \le m$, 
    $$ \widehat{\rho}(\a_t^{(\bullet)},(\m^N\a_t)^{\bullet} ) = \frac{N+\alpha(\a_t)}{\widehat{\alpha}(\a_t)}=\frac{(N+t)(m-t+1)}{m}.$$
\end{theorem}
\begin{proof}
By \ref{P3}, we get $\alpha(\a_1)=1$, and hence, by \ref{P2}, $\alpha(\a_t)=t.$ Using \Cref{wald-family}, we get that $$\frac{N+\alpha(\a_t)}{\widehat{\alpha}(\a_t)}=\frac{(N+t)(m-t+1)}{m}.$$ Next, it follows from  \cite[Corollary 2.9]{HKNN23} that $\widehat{\rho}(\a_t^{(\bullet)},(\m^N\a_t)^{\bullet} ) \ge \frac{\alpha(\m^N\a_t)}{\widehat{\alpha}(\a_t)}=\frac{N+t}{\widehat{\alpha}(\a_t)}.$   So, it is enough to show that 
    \begin{align*}
        \widehat{\rho}(\a_t^{(\bullet)},(\m^N\a_t)^{\bullet} )\leq \frac{(N+t)(m-t+1)}{m}.
    \end{align*}
    Let $s_k=k(N+t)(m-t+1), r_k=km$ for $k\in\mathbb{N}$. Notice that $$\lim_{k \to \infty} s_k = \lim_{k \to \infty} r_k =\infty \text{ and }\lim_{k\rightarrow\infty}\frac{s_k}{r_k}=\frac{(N+t)(m-t+1)}{m}.$$
We claim that 
$    \a_t^{(s_k)}\subseteq \overline{(\m^N\a_t)^{r_k}}$ for all $k \ge 1$. 

Since $\RR_s(\a_t) =R[\a_tT,\a_{t+1}T^2,\ldots,\a_mT^{m-t+1}]$ (see \ref{P1}), \begin{align*}
        \a_t^{(s_k)}=\sum_{\substack{a_i\in \NN,\\\sum_{i=1}^{m-t+1}ia_i=s_k}}\a_t^{a_1}\cdots \a_m^{a_{m-t+1}}.
    \end{align*}
     Therefore, it is sufficient to prove that for all non-negative integers $a_1,\ldots, a_{m-t+1}$ with $\sum_{i=1}^{m-t+1}ia_i=s_k$, $\a_t^{a_1}\cdots \a_m^{a_{m-t+1}} \subseteq \overline{(\m^N\a_t)^{r_k}}$. In fact, due to \ref{P3},  it is sufficient to prove that $\a_t^{a_1}\cdots \a_m^{a_{m-t+1}} \subseteq \a_{j}^{(r_k(t-j+1))}$ for all $ 2\le j \le t$ and $\a_t^{a_1}\cdots \a_m^{a_{m-t+1}} \subseteq \a_{1}^{(r_k(N+t))}.$ Let $a_1,\ldots,a_{m-t+1}$ be non-negative integers with $\sum_{i=1}^{m-t+1}ia_i=s_k$. 

    First, suppose $j=t$. Verifying that $s_k \ge r_k$ for all $k$ is easy. Therefore, $\a_t^{(s_k)} \subseteq \a_t^{(r_k)}$ for all $k$, and hence, $\a_t^{a_1}\cdots \a_m^{a_{m-t+1}} \subseteq \a_t^{(s_k)} \subseteq \a_t^{(r_k)}=\a_{t}^{(r_k(t-t+1))}$.

    Next, suppose $2 \le j <t.$ Since $\RR_s(\a_j) =R[\a_jT,\a_{j+1}T^2,\ldots,\a_mT^{m-j+1}]$ (see \ref{P1}), \begin{align*}
        \a_j^{(r_k(t-j+1))}=\sum_{\substack{b_i\in \NN,\\\sum_{i=1}^{m-j+1}ib_i=r_k(t-j+1)}}\a_j^{b_1}\cdots \a_m^{b_{m-j+1}}.
    \end{align*} 
    Now, take $b_i=0$ for $1 \le i \le t-j$ and $b_i =a_{i-t+j}$ for $t-j+1 \le i \le m-j+1$. Then, \begin{align*}
        \sum_{i=1}^{m-j+1} ib_i &= \sum_{i=t-j+1}^{m-j+1} ia_{i-t+j} =  \sum_{i=1}^{m-t+1} (i+t-j)a_i \\ &= \sum_{i=1}^{m-t+1} ia_i+(t-j)\sum_{i=1}^{m-t+1} a_i  =s_k+(t-j)\sum_{i=1}^{m-t+1}a_i.
    \end{align*}
    Let, if possible, $\sum_{i=1}^{m-j+1} ib_i \le (t-j+1)r_k -1$, then    \begin{align*}
        (t-j)\sum_{i=1}^{m-t+1}a_i & =\sum_{i=1}^{m-j+1} ib_i -s_k \le  (t-j+1)r_k -1 -s_k .
    \end{align*} Furthermore, \begin{align*}
        \sum_{i=1}^{m-t+1}a_i \le r_k +\frac{r_k-s_k-1}{t-j}.
    \end{align*} This implies that 
    \begin{align*}
        s_k =\sum_{i=1}^{m-t+1}i a_i \le (m-t+1) \sum_{i=1}^{m-t+1} a_i  \le (m-t+1) \left( r_k +\frac{r_k-s_k-1}{t-j}\right).
    \end{align*} 
Now, $s_k \le (m-t+1) \left( r_k +\frac{r_k-s_k-1}{t-j}\right)$ if and only if   $kN(t-j)+s_k -r_k+1 \le k(t-j)(m-t)$ if and only if  $kN(t-j)+k(m-t)(t -1) +1 \le k(t-j)(m-t)$ which is a contradiction as $t-j \le t-1.$ Thus,  $\sum_{i=1}^{m-j+1} ib_i \ge (t-j+1)r_k$, and hence, $$\a_t^{a_1}\cdots \a_m^{a_{m-t+1}} = \a_j^{b_1}\cdots \a_m^{b_{m-j+1}} \subseteq \a_j^{(\sum_{i=1}^{m-j+1}ib_i)} \subseteq \a_j^{(r_k(t-j+1))}.$$

For $j=1$, $\a_1$ is generated by linear forms, which means that $\a_1$ is a complete intersection. Consequently, we have $\a_1^{(s)}=\a_1^s$ for all $s$. According to \ref{P1}, it follows that  $\RR_s(\a_1)=R[\a_1T,\ldots,\a_mT^m]$, implying  that  $\a_j \subseteq  \a_1^{(j)}=\a_1^j$ for all $2 \le j \le m.$ Using these containments, we obtain  $\a_t^{a_1} \cdots \a_m^{a_{m-t+1}} \subseteq \a_1^{ta_1+\cdots+ma_{m-t+1}} =\a_1^{\sum_{i=1}^{m-t+1}(t-1+i)a_i}.$ Therefore,  to prove that $\a_t^{a_1}\cdots \a_m^{a_{m-t+1}} \subseteq \a_{1}^{(r_k(N+t))}=\a_1^{Nr_k+tr_k}$,   it is sufficient to show that $ {\sum_{i=1}^{m-t+1}(t-1+i)a_i} \ge {Nr_k+tr_k}$. Consider, \begin{align*}
    \sum_{i=1}^{m-t+1}(t-1+i)a_i =(t-1) \sum_{i=1}^{m-t+1} a_i +\sum_{i=1}^{m-t+1} ia_i =(t-1)\sum_{i=1}^{m-t+1} a_i +s_k \ge (N+t)r_k, 
\end{align*} where the last inequality will hold if and only if $\sum_{i=1}^{m-t+1} a_i \ge \frac{(N+t)r_k -s_k}{t-1}=k(N+t)$. Specifically, if $\sum_{i=1}^{m-t+1} a_i \le k(N+t)-1$, then  $s_k=\sum_{i+1}^{m-t+1}ia_i \le (m-t+1)\sum_{i=1}^{m-t+1} a_i \le (m-t+1)(k(N+t)-1)$ which leads to a contradiction. Therefore, we must have $\sum_{i=1}^{m-t+1} a_i \ge \frac{(N+t)r_k -s_k}{t-1}=k(N+t)$. This implies that  $ {\sum_{i=1}^{m-t+1}(t-1+i)a_i} \ge {Nr_k+tr_k}$.  Thus, $\a_t^{(s_k)} \subseteq \a_1^{(Nr_k+tr_k)}$ for all $k.$
Hence,  $\a_t^{(s_k)}\subseteq \overline{(\m^N\a_t)^{r_k}}$ for all $k$.  It follows from \cite[Lemma 3.1]{HKNN23} (also see \cite[Lemma 4.1]{MCM19}) that $\widehat{\rho}(\a_t^{(\bullet)},\overline{(\m^N\a_t)^{\bullet}} ) \le \frac{N+\alpha(\a_t)}{\widehat{\alpha}(\a_t)}.$  By \cite[Theorem 3.2]{HKNN23} (also see \cite[Proposition 4.2]{MCM19}), we know that $\widehat{\rho}(\a_t^{(\bullet)},\overline{(\m^N\a_t)^{\bullet}} ) =\widehat{\rho}(\a_t^{(\bullet)},(\m^N\a_t)^{\bullet} ).$ Hence, the assertion follows. 
\end{proof}

\begin{theorem}\label{asym-res}
   Assume $R=\k[x_1,\dots,x_n]$ and $\a_1,\dots,\a_m$ be a sequence of non-zero proper homogeneous ideals in $R$ satisfying \ref{P1} and \ref{P2}. Assume that \ref{P3} holds for $N=0$. Then, for all $1 \le t \le m$, 
    $$ \widehat{\rho}(\a_t) = \frac{\alpha(\a_t)}{\widehat{\alpha}(\a_t)}=\frac{t(m-t+1)}{m}.$$
\end{theorem}
\begin{proof}
The proof follows immediately by putting $N=0$ in \Cref{asym-res-HH1}. 
\end{proof}

In this example, we will showcase how simple techniques can be used to recover some of the results already available in the literature. Additionally, in the next section, we will present a vast range of ideal families and examine their containment problem-related invariants. 
\begin{example}\label{star config example}
    Let $f_1,\ldots,f_m$ be a $R$-regular sequence of linear forms in $R=\k[x_0,\ldots,x_n]$. Let $I_{m,c}$ denote the ideal of star configuration in $\mathbb{P}^n$ of codimension $c$, i.e., $$I_{m,c} = \bigcap_{1 \le i_1<i_2 \cdots <i_c\le m} (f_{i_1},\ldots, f_{i_c}).$$ Then, it follows from \cite[Theorem 4.8]{GeramitaHarbourneMigliore} and \cite[Section 3]{GHMN} that for all $1 \le c \le m$,  $$I_{m,c}^s = \bigcap_{j=0}^{m-c} I_{m,c+j}^{(s(j+1))} = \bigcap_{j=c}^{m} I_{m,j}^{(s(j-c+1))}. $$ Also, it follows from \cite[Proposition 4.6]{HHTV} and \cite[Section 3]{GHMN} that for all $1 \le c \le m$, $$\RR_s(I_{m,c}) =R[I_{m,c}T, I_{m,c-1}T^2, \ldots, I_{m,1}T^{c}].$$ Now, set $\a_t=I_{m,m-t+1}$ for all $1 \le t \le m.$ Then, $\a_1,\ldots,\a_m$ is a sequence of non-zero ideals that satisfies \ref{P1}, \ref{P2} and \ref{P3} for any non-negative integer $N$. Hence, using \Cref{asym-res-HH1}, we recover well-known results \cite[Theorem C]{LM15} and \cite[Theorem 4.8(1)]{GHMN} about the resurgence number of monomial/linear star configuration $$\rho(I_{m,c}) =\widehat{\rho}(I_{m,c})=\frac{\alpha(I_{m,c})}{\widehat{\alpha}(I_{m,c})}=\frac{c(m-c+1)}{m}.$$
\end{example}

One of our primary objectives is to study Question A for classical varieties. If $j=iN$ for some $N \in \NN$ in Question A, then  \Cref{asym-res-HH1}, establishes  that for every $s \ge \frac{(N+t)(m-t+1)i}{m}$, infinitely many $k$, we have $\a_t^{(sk)} \subseteq \m^{iNk}\a_t^{ik}$. Interestingly, for the classical varieties, we will see in the next section that  $\a_t^{(s)} \subseteq \m^{iN}\a_t^{i}$ for every $s \ge \frac{(N+t)(m-t+1)i}{m}$. So, a natural question seeks the consequences of assuming  $j$ is not a multiple of $i$ in Question A. In the next section, we answer this affirmatively for the case of classical varieties. 

\begin{theorem}\label{asym-res-HH}
  Assume $R=\k[x_1,\dots,x_n]$ and $\a_1,\dots,\a_m$ be a sequence of non-zero proper homogeneous ideals in $R$ satisfying \ref{P1} and \ref{P2}. Let $N$ be a non-negative integer so that \ref{P4} holds. Then, for all $1 \le t \le m$, 
    $$ \widehat{\rho}(\a_t^{(\bullet)},\overline{\m^N(\a_t)^{\bullet}} ) = \frac{\widehat{\alpha}(\overline{\m^N(\a_t)^\bullet)}}{\widehat{\alpha}(\a_t)}=\frac{t(m-t+1)}{m}=\frac{\alpha(\a_t)}{\widehat{\alpha}(\a_t)}.$$
\end{theorem}
\begin{proof}
By \ref{P4}, $\alpha(\a_1)=1$, and therefore, \ref{P2} implies that $\alpha(\a_t)=t.$ Using \Cref{wald-family}, we get that $$\frac{\alpha(\a_t)}{\widehat{\alpha}(\a_t)}=\frac{t(m-t+1)}{m}.$$ Note that for all $r \in \NN$, $\alpha(\overline{\m^N\a_t^r})=\alpha(\m^N\a_t^r)=N+\alpha(\a_t^r)=N+r\alpha(\a_t).$ Thus, $\widehat{\alpha}(\overline{\m^N(\a_t)^\bullet})= \lim\limits_{r \to \infty} \frac{\alpha(\m^N\a_t^r)}{r}=\alpha(\a_t).$ Next, it follows from  \cite[Corollary 2.9]{HKNN23} that $$\widehat{\rho}(\a_t^{(\bullet)},\overline{\m^N(\a_t)^{\bullet}} ) \ge \frac{\widehat{\alpha}(\overline{\m^N(\a_t)^\bullet})}{\widehat{\alpha}(\a_t)}=\frac{\alpha(\a_t)}{\widehat{\alpha}(\a_t)}=\frac{t(m-t+1)}{m}.$$   So, it is enough to show that 
    \begin{align*}
        \widehat{\rho}(\a_t^{(\bullet)},\overline{\m^N(\a_t)^{\bullet} })\leq \frac{t(m-t+1)}{m}.
    \end{align*}
    Let $s_k=\left(kt+\ceil{\frac{N}{m}}\right)(m-t+1), r_k=km$ for $k\in\mathbb{N}$. Notice that $$\lim_{k \to \infty} s_k = \lim_{k \to \infty} r_k =\infty \text{ and }\lim_{k\rightarrow\infty}\frac{s_k}{r_k}=\frac{t(m-t+1)}{m}.$$
We claim that 
$    \a_t^{(s_k)}\subseteq \overline{\m^N\a_t^{r_k}}$ for all $k \ge 1$. 

Since $\RR_s(\a_t) =R[\a_tT,\a_{t+1}T^2,\ldots,\a_mT^{m-t+1}]$ (see \ref{P1}), \begin{align*}
        \a_t^{(s_k)}=\sum_{\substack{a_i\in \NN,\\\sum_{i=1}^{m-t+1}ia_i=s_k}}\a_t^{a_1}\cdots \a_m^{a_{m-t+1}}.
    \end{align*}
     Therefore, it is sufficient to prove that for all non-negative integers $a_1,\ldots, a_{m-t+1}$ with $\sum_{i=1}^{m-t+1}ia_i=s_k$, $\a_t^{a_1}\cdots \a_m^{a_{m-t+1}} \subseteq \overline{\m^N\a_t^{r_k}}$. In fact, due to \ref{P4},  it is sufficient to prove that $\a_t^{a_1}\cdots \a_m^{a_{m-t+1}} \subseteq \a_{j}^{(r_k(t-j+1))}$ for all $ 2\le j \le t$ and $\a_t^{a_1}\cdots \a_m^{a_{m-t+1}} \subseteq \a_{1}^{(N+tr_k)}.$ Let $a_1,\ldots,a_{m-t+1}$ be non-negative integers with $\sum_{i=1}^{m-t+1}ia_i=s_k$. 

    First, suppose $j=t$. Verifying that $s_k \ge r_k$ for all $k$ is easy. Therefore, $\a_t^{(s_k)} \subseteq \a_t^{(r_k)}$ for all $k$, and hence, $\a_t^{a_1}\cdots \a_m^{a_{m-t+1}} \subseteq \a_t^{(s_k)} \subseteq \a_t^{(r_k)}=\a_{t}^{(r_k(t-t+1))}$.

    Next, suppose $2 \le j <t.$ Since $\RR_s(\a_j) =R[\a_jT,\a_{j+1}T^2,\ldots,\a_mT^{m-j+1}]$ (see \ref{P1}), \begin{align*}
        \a_j^{(r_k(t-j+1))}=\sum_{\substack{b_i\ge 0,\\\sum_{i=1}^{m-j+1}ib_i=r_k(t-j+1)}}\a_j^{b_1}\cdots \a_m^{b_{m-j+1}}.
    \end{align*} 
    Now, take $b_i=0$ for $1 \le i \le t-j$ and $b_i =a_{i-t+j}$ for $t-j+1 \le i \le m-j+1$. Then, \begin{align*}
        \sum_{i=1}^{m-j+1} ib_i &= \sum_{i=t-j+1}^{m-j+1} ia_{i-t+j} =  \sum_{i=1}^{m-t+1} (i+t-j)a_i \\ &= \sum_{i=1}^{m-t+1} ia_i+(t-j)\sum_{i=1}^{m-t+1} a_i  =s_k+(t-j)\sum_{i=1}^{m-t+1}a_i.
    \end{align*}
    Let, if possible, $\sum_{i=1}^{m-j+1} ib_i \le (t-j+1)r_k -1$, then    \begin{align*}
        (t-j)\sum_{i=1}^{m-t+1}a_i & =\sum_{i=1}^{m-j+1} ib_i -s_k \le  (t-j+1)r_k -1 -s_k .
    \end{align*} Furthermore, \begin{align*}
        \sum_{i=1}^{m-t+1}a_i \le r_k +\frac{r_k-s_k-1}{t-j}.
    \end{align*} This implies that 
    \begin{align*}
        s_k =\sum_{i=1}^{m-t+1}i a_i \le (m-t+1) \sum_{i=1}^{m-t+1} a_i  \le (m-t+1) \left( r_k +\frac{r_k-s_k-1}{t-j}\right).
    \end{align*} 
Now, $s_k \le (m-t+1) \left( r_k +\frac{r_k-s_k-1}{t-j}\right)$ if and only if   $\ceil{\frac{N}{m}}(t-j)+s_k -r_k+1 \le k(t-j)(m-t)$ if and only if  $k(m-t)(t -1)+\ceil{\frac{N}{m}}(m-j+1) +1 \le k(t-j)(m-t)$ which is a contradiction as $t-j \le t-1.$ Thus,  $\sum_{i=1}^{m-j+1} ib_i \ge (t-j+1)r_k$, and hence, $$\a_t^{a_1}\cdots \a_m^{a_{m-t+1}} = \a_j^{b_1}\cdots \a_m^{b_{m-j+1}} \subseteq \a_j^{(\sum_{i=1}^{m-j+1}ib_i)} \subseteq \a_j^{(r_k(t-j+1))}.$$ 

For $j=1$, $\a_1$ is generated by linear forms, which means that $\a_1$ is a complete intersection. Consequently, we have $\a_1^{(s)}=\a_1^s$ for all $s$. According to \ref{P1}, it follows that  $\RR_s(\a_1)=R[\a_1T,\ldots,\a_mT^m]$, implying  that  $\a_j \subseteq  \a_1^{(j)}=\a_1^j$ for all $2 \le j \le m.$ Using these containments, we obtain  $\a_t^{a_1} \cdots \a_m^{a_{m-t+1}} \subseteq \a_1^{ta_1+\cdots+ma_{m-t+1}} =\a_1^{\sum_{i=1}^{m-t+1}(t-1+i)a_i}.$ Therefore,  to prove that $\a_t^{a_1}\cdots \a_m^{a_{m-t+1}} \subseteq \a_{1}^{(N+tr_k)}=\a_1^{N+tr_k}$,   it is sufficient to show that $ {\sum_{i=1}^{m-t+1}(t-1+i)a_i} \ge {N+tr_k}$. Consider, \begin{align*}
    \sum_{i=1}^{m-t+1}(t-1+i)a_i =(t-1) \sum_{i=1}^{m-t+1} a_i +\sum_{i=1}^{m-t+1} ia_i =(t-1)\sum_{i=1}^{m-t+1} a_i +s_k \ge N+tr_k, 
\end{align*} where the last inequality will hold if and only if $\sum_{i=1}^{m-t+1} a_i \ge \frac{N+tr_k -s_k}{t-1}=kt+\ceil{\frac{N}{m}} +\frac{N-m\ceil{\frac{N}{m}}}{t-1}$. Specifically, if $\sum_{i=1}^{m-t+1} a_i \le kt+\ceil{\frac{N}{m}} +\frac{N-m\ceil{\frac{N}{m}}}{t-1}-1$, then  $$s_k=\sum_{i+1}^{m-t+1}ia_i \le (m-t+1)\sum_{i=1}^{m-t+1} a_i \le (m-t+1)\left( kt+\ceil{\frac{N}{m}} +\frac{N-m\ceil{\frac{N}{m}}}{t-1}-1\right)$$ which leads to a contradiction. Therefore, we must have $\sum_{i=1}^{m-t+1} a_i \ge \frac{N+tr_k -s_k}{t-1}$. This implies that  $ {\sum_{i=1}^{m-t+1}(t-1+i)a_i} \ge {N+tr_k}$. Thus, $\a_t^{(s_k)} \subseteq \a_1^{(N+tr_k)}$ for all $k.$ Hence,   $\a_t^{(s_k)}\subseteq \overline{\m^N\a_t^{r_k}}$ for all $k$.  It follows from \cite[Lemma 3.1]{HKNN23} (also see \cite[Lemma 4.1]{MCM19}) that $\widehat{\rho}(\a_t^{(\bullet)},\overline{\m^N(\a_t)^{\bullet}} ) \le \frac{\alpha(\a_t)}{\widehat{\alpha}(\a_t)}.$  
\end{proof}
In \cite{HKNN23}, analogs of  \Cref{asym-res-HH,asym-res-HH1} for $\widehat{\rho}(\a_\bullet,\b_\bullet)$ are known, under the condition that the  Rees algebra of the second family $\mathcal{R}(\b_\bullet)$ is Noetherian. Conversely, the authors in \cite{HKNN23} demonstrate that these results fail when the Rees algebra $\mathcal{R}(\b_\bullet)$ is not Noetherian. Interestingly, the second family $\b_\bullet$ in  \Cref{asym-res-HH} does not possess Noetherian Rees algebras, and classical varieties constitute a broad class satisfying all conditions imposed in this section. Consequently, our results provide substantial examples, including classical varieties, that affirmatively answer Question 2.19 from \cite{HKNN23}.

\section{Applications}\label{Application sections}
In this section, we justify the need to study the family of ideals $\{\a_t\}$, which satisfies \ref{P1}--\ref{P4}. As we mentioned before, all the classical varieties exhibit these properties. The classical varieties are, in fact, a large class of ideals and showcase the need to study such families. Through out this section, $\k$ denote a field of  $\Char(\k)=0$ or a $F$-finite field of $\Char(\k)=p>0$. A field $\k$ of characteristic  $p > 0$ is called $F$-finite if the degree of the field extension $ [\k:\k^p]$ is finite.

    \subsection{Ideals of minors of a generic matrix:}   Let $X= [x_{ij}]$ be a $m \times n$ generic matrix of variables with  $m \le n.$ Let $R= \k[X]$ and $I_t(X)$ be the ideal of $t \times t$ minors of $X$. We set $\a_t=I_t(X)$ for $1 \le t \le m$. \ref{P1} is satisfied as the symbolic Rees algebra $$\mathcal{R}_s(I_t(X))= R[I_t(X)T,I_{t+1}(X)T^2,\dots, I_m(X)T^{m-t+1}]$$ (by \cite[Theorem 10.4]{BrunsVetter88}). Clearly, $\a_1 =\mathfrak{m}$, the homogeneous maximal ideal of $R$ and $\alpha(\a_t) =t$ for all $ 1\le t \le m$. So, \ref{P2} is easily satisfied. It follows from \cite[Theorem 1.3]{W91} and \cite[Theorem 10.9]{BrunsVetter88} that for all $ 1 \le t \le m$, $N \in \NN$ and $s\ge 1$,  \begin{align*}
        \overline{\m^{N}I_t(X)^s}= I_1(X)^{N+ts} \bigcap \left(\bigcap_{j=2}^t I_j(X)^{(s(t-j+1))}\right).    
        \end{align*} Thus, $I_1(X),\ldots,I_m(X)$ satisfy \ref{P3} and \ref{P4} for any non-negative integer $N$. 
\begin{theorem}\label{generic-thm}
    Let $X= [x_{ij}]$ be a $m \times n$ generic matrix of variables with  $m \le n.$ Then, 
    \begin{enumerate}
        \item For any $k \ge 1$ and $0\leq l\leq m-t$, 
        $\alpha\left(I_t(X)^{(k(m-t+1)-l)}\right)=km-l.$
Furthermore, the Waldschmidt constant is given by the equation $$\widehat{\alpha}(I_t(X))=\frac{m}{m-t+1}.$$
\item  For $N \in \NN$ and $1<t<m$, the following are equivalent \begin{enumerate} 
 \item   for all  $s,r\ge 1$,
    \begin{align*}
        \frac{\alpha\left(I_t(X)^{(s)}\right)}{s}\geq\frac{\alpha\left(I_t(X)^{(r)}\right)+N-1}{r+N-1} \text{ (Demailly-type bound)}
    \end{align*} 
    \item for all  $s\ge 1$,
    \begin{align*}
        \frac{\alpha\left(I_t(X)^{(s)}\right)}{s}\geq\frac{\alpha\left(I_t(X)\right)+N-1}{N} \text{   (Chudnovsky-type bound)}
    \end{align*}
    \item    $N\geq m-t+1$. \end{enumerate}
    \item For $N \in \NN$, 
    $$ \widehat{\rho}\left(I_t(X)^{(\bullet)}, \left(\m^NI_t(X)\right)^\bullet\right) = \frac{N+\alpha(I_t(X))}{\widehat{\alpha}(I_t(X))}=\frac{(N+t)(m-t+1)}{m}.$$
    Furthermore, if $(\min\{t,m-t\})!$ is invertible in $\k$, then $$ \rho\left(I_t(X)^{(\bullet)}, \left(\m^NI_t(X)\right)^\bullet\right)=  \frac{N+\alpha(I_t(X))}{\widehat{\alpha}(I_t(X))}=\frac{(N+t)(m-t+1)}{m}.$$
    \item For $N \in \NN$, 
    $$ \widehat{\rho}\left(I_t(X)^{(\bullet)}, \overline{\m^N\left(I_t(X)\right)^\bullet}\right) = \frac{\alpha(I_t(X))}{\widehat{\alpha}(I_t(X))}=\frac{t(m-t+1)}{m}.$$
    Furthermore, if $(\min\{t,m-t\})!$ is invertible in $\k$, then $$ \widehat{\rho}\left(I_t(X)^{(\bullet)}, \m^N\left(I_t(X)\right)^\bullet\right)=  \frac{\alpha(I_t(X))}{\widehat{\alpha}(I_t(X))}=\frac{t(m-t+1)}{m}.$$
    \end{enumerate}
\end{theorem}
\begin{proof}
    The proof follows from \Cref{wald-family}, \Cref{chudnovsky and demailly bounds}, \Cref{asym-res} and \Cref{asym-res-HH} except the second parts of $(3-4)$. When $(\min\{t,m-t\})!$ is invertible in $\k$, by \cite[Theorem 1.1]{W91} and \cite{BrunsVetter88}, $\overline{\m^rI_t(X)^s} =\m^rI_t(X)^s$ for all $s,r$. Therefore, the second part of $(4)$ follows immediately. Using \cite[Corollary 4.7]{HKNN23}, we get $\rho\left(I_t(X)^{(\bullet)}, \left(\m^NI_t(X)\right)^\bullet\right)= \widehat{\rho}\left(I_t(X)^{(\bullet)},  \left(\m^NI_t(X)\right)^\bullet\right)$, i.e., the second part of $(3)$ follows.  
\end{proof}
\begin{remark}\label{remakr on demaily applications}
A result analogous to the Statement $(2)$ of the above theorem is presented in \cite[Theorem 3.8]{BGHN22}. But notice that the above result recovers the Demailly type bound presented in \cite[Theorem 3.8]{BGHN22} as $N$ is much smaller than the big height in this case.    
\end{remark}

\subsection{Ideal of minors of a ladder of generic matrix:}
    Let $X=[x_{ij}]$ be a  $r\times n$ generic matrix of variables with $r \le n$, $S=\k[X]$ and $I_t(X)$ denote the ideal of $t \times t$ minors of $X$. A subset $L$ is called a ladder of $X$ if $x_{ij},x_{rs}\in L$ with $i<r,j<s$ then $x_{is},x_{rj}\in L$ (see \cite{HerzogTrung92} for general definitions). Set $R=\k[L]$ and $I_t(L)=I_t(X)\cap R$ denote the ideal generated by minors whose elements belong to $L$. We set $\a_t=I_t(L)$ for $ 1\le t \le m$, where $m$ is the largest possible integer such that all the entries of a $m \times m$ submatrix of $X$ belong to $L$.  Using \cite[Theorem 4.1]{Bruns-Conca-98}, 
    \begin{align*}
     \mathcal{R}_s(I_t(L))= R[I_t(L)T,I_{t+1}(L)T^2,\dots, I_m(L)T^{m-t+1}],   
    \end{align*}
     and hence \ref{P1} is satisfied. Clearly, $\a_1=\m \cap R$ which is the maximal homogeneous ideal of $R$ and $\alpha(\a_t) =t$ for all $ 1\le t \le m$. Therefore,  \ref{P2} is also satisfied.
\begin{theorem}
    Let $X= [x_{ij}]$ be a $r \times n$ generic matrix of variables with  $r \le n.$ Let $m$ be the largest possible integer such that all the entries of a $m \times m$ submatrix of $X$ belong to $L$.  Then, 
    \begin{enumerate}
        \item For any $k \ge 1$ and $0 \le l \le m-t$, $
        \alpha\left(I_t(L)^{(k(m-t+1)-l)}\right)=km-l.$
Furthermore, the Waldschmidt constant is given by the equation $$\widehat{\alpha}(I_t(L))=\frac{m}{m-t+1}.$$
\item  For $N \in \NN$ and $1<t<m$, the following are equivalent
\begin{enumerate} \item   for all  $s,r\ge 1$,
    \begin{align*}
        \frac{\alpha\left(I_t(L)^{(s)}\right)}{s}\geq\frac{\alpha\left(I_t(L)^{(r)}\right)+N-1}{r+N-1} \text{ (Demailly-type bound)}
    \end{align*} 
    \item for all  $s\ge 1$,
    \begin{align*}
        \frac{\alpha\left(I_t(L)^{(s)}\right)}{s}\geq\frac{\alpha\left(I_t(L)\right)+N-1}{N} \text{ (Chudnovsky-type bound)}
    \end{align*} \item    $N\geq m-t+1$. \end{enumerate}
    \end{enumerate}
\end{theorem}
\begin{proof}
The proof follows from \Cref{wald-family}, \Cref{chudnovsky and demailly bounds}.
\end{proof}

    \subsection{Ideals of minors of a symmetric matrix:}   Let $Y=[y_{ij}]$ be $m \times m$ generic symmetric matrix of variables. Let $R= \k[Y]$ and $I_t(Y)$ be the ideal of $t \times t$ minors of $Y$. We set $\a_t=I_t(Y)$ for $1 \le t \le m$. In this case, \ref{P1} is satisfied as
    $$\mathcal{R}_s(I_t(Y))= R[I_t(Y)T,I_{t+1}(Y)T^2,\dots, I_m(Y)T^{m-t+1}]$$ (by \cite[Proposition 4.3]{JeffriesMontanoVarbaro15}). Clearly, $\a_1 =\mathfrak{m}$, the homogeneous maximal ideal of $R$ and $\alpha(\a_t) =t$ for all $ 1 \le t \le m$. So, \ref{P2} is easily satisfied. It follows from \cite[Theorem 2.7]{IM16} that for all $ 1 \le t \le m$, $N \in \NN$ and $s \ge 1$,  \begin{align*}
        \overline{\m^{N}I_t(Y)^s}= I_1(Y)^{N+ts} \bigcap \left( \bigcap_{j=2}^t I_j(Y)^{(s(t-j+1))} \right).    
        \end{align*} Thus, $I_1(Y),\ldots,I_m(Y)$ satisfy \ref{P3} and \ref{P4} for any non-negative integer $N$. 
\begin{theorem}
    Let $Y= [y_{ij}]$ be a $m \times m$ generic symmetric matrix of variables. Then, 
    \begin{enumerate}
        \item For $k \ge 1$ and $0 \le l \le m-t$, $
        \alpha\left(I_t(Y)^{(k(m-t+1)-l)}\right)=km-l.$
Furthermore, the Waldschmidt constant is given by the equation $$\widehat{\alpha}(I_t(Y))=\frac{m}{m-t+1}.$$
\item For $N \in \NN$ and $1<t<m$, the following are equivalent \begin{enumerate}  \item  for all  $s,r\ge 1$,
    \begin{align*}
        \frac{\alpha\left(I_t(Y)^{(s)}\right)}{s}\geq\frac{\alpha\left(I_t(Y)^{(r)}\right)+N-1}{r+N-1} \text{ (Demailly-type bound) }
    \end{align*} 
    \item for all  $s\ge 1$,
    \begin{align*}
        \frac{\alpha\left(I_t(Y)^{(s)}\right)}{s}\geq\frac{\alpha\left(I_t(Y)\right)+N-1}{N} \text{ (Chudnovsky-type bound)}
    \end{align*} 
    \item   $N\geq m-t+1$. \end{enumerate}
    \item For $N \in \NN$, 
    $$ \widehat{\rho}\left(I_t(Y)^{(\bullet)}, \left(\m^NI_t(Y)\right)^\bullet\right) = \frac{N+\alpha(I_t(Y))}{\widehat{\alpha}(I_t(Y))}=\frac{(N+t)(m-t+1)}{m}.$$
    Furthermore, if $\Char(\k) =0$ or $\Char(\k) > \min\{t,m-t\}$, then $$ \rho\left(I_t(Y)^{(\bullet)}, \left(\m^NI_t(Y)\right)^\bullet\right)=  \frac{N+\alpha(I_t(Y))}{\widehat{\alpha}(I_t(Y))}=\frac{(N+t)(m-t+1)}{m}.$$
    \item For $N \in \NN$, 
    $$ \widehat{\rho}\left(I_t(Y)^{(\bullet)}, \overline{\m^N\left(I_t(Y)\right)^\bullet}\right) = \frac{\alpha(I_t(Y))}{\widehat{\alpha}(I_t(Y))}=\frac{t(m-t+1)}{m}.$$
    Furthermore, if $\Char(\k) =0$ or $\Char(\k) > \min\{t,m-t\}$, then $$ \widehat{\rho}\left(I_t(Y)^{(\bullet)}, \m^N\left(I_t(Y)\right)^\bullet\right)=  \frac{\alpha(I_t(Y))}{\widehat{\alpha}(I_t(Y))}=\frac{t(m-t+1)}{m}.$$
    \end{enumerate}
\end{theorem}
\begin{proof}
The proof follows from \Cref{wald-family}, \Cref{chudnovsky and demailly bounds}, \Cref{asym-res} and \Cref{asym-res-HH} except the second part of $(3-4)$. When $\Char(\k) =0$ or $\Char(\k) > \min\{t,m-t\}$, by \cite[Proof of Theorem 2.7]{IM16} and \cite[Theorem 1.3]{W91}, $\overline{\m^{r}I_t(Y)^s} =\m^rI_t(Y)^s$ for all $s, r$. Therefore, the second part of $(4)$ follows immediately. Using \cite[Corollary 4.7]{HKNN23}, we get $\rho\left(I_t(Y)^{(\bullet)}, \left(\m^NI_t(Y)\right)^\bullet\right)= \widehat{\rho}\left(I_t(Y)^{(\bullet)},  \left(\m^NI_t(Y)\right)^\bullet\right)$, i.e., the second part of $(3)$ follows.
\end{proof}
The remark analogous to \Cref{remakr on demaily applications} is true in this case too.

    \subsection{Ideals of Pfaffians of a skew-symmetric matrix:}  Let $Z=[z_{ij}]$ be $m \times m$ generic skew-symmetric  matrix of variables. Let $R= \k[Z]$ and $\Pf_{2t}(Z)$ be the ideal of $2t \times 2t$ pfaffians of $Z$. A pfaffian of order $2t$ is the square root of the determinant of a $2t\times 2t$ submatrix of $Z$ obtained from the rows $\{i_1,\dots,i_{2t}\}$ and the columns $\{i_1,\dots,i_{2t}\}$. We set $\a_t=\Pf_{2t}(Z)$ for $1 \le t \le \floor{\frac{m}{2}}$. In this case, \ref{P1} is satisfied as
    $$\mathcal{R}_s(\Pf_{2t}(Z))= R[\Pf_{2t}(Z)T,\Pf_{2t+2}(Z)T^2,\dots, \Pf_{2 \floor{\frac{m}{2}}}(Z)T^{\floor{\frac{m}{2}}-t+1}]$$ (by \cite[Theorem 2.1]{Negri96} and \cite[Proposition 4.5]{JeffriesMontanoVarbaro15}; see also \cite[Section 3]{Baetica98}). Clearly, $\a_1 =\Pf_{2}(Z) =\mathfrak{m}$, the homogeneous maximal ideal of $R$ and $\alpha(\a_t) = t$ for all $ 1 \le t \le \floor{\frac{m}{2}}.$ So, \ref{P2} is easily satisfied. Also, it follows from \cite[Theorem 2.10]{IM16} that for all $1 \le t \le \floor{\frac{m}{2}},$ $N \in \NN$ and $s\ge 1$
     \begin{align*}
        \overline{\m^{N}\Pf_{2t}(Z)^s}=\Pf_2(Z)^{N+ts} \bigcap \left( \bigcap_{j=2}^t \Pf_{2j}(Z)^{(s(t-j+1))}\right).    
        \end{align*} Thus, $\Pf_{2}(Z),\ldots,\Pf_{2\floor{\frac{m}{2}}}(Z)$ satisfy \ref{P3} and \ref{P4} for any non-negative integer $N$.
        
\begin{theorem}
    Let $Z= [z_{ij}]$ be a $m \times m$ generic skew-symmetric matrix of variables. Then, 
    \begin{enumerate}
        \item For $k \ge 1$ and $0 \le l \le \floor{\frac{m}{2}}-t$, 
        $\alpha\left(\Pf_{2t}(Z)^{(k(\floor{\frac{m}{2}}-t+1)-l)}\right)=
                k\floor{\frac{m}{2}}-l.$
Furthermore, the Waldschmidt constant is given by the equation $$\widehat{\alpha}(\Pf_{2t}(Z))=\frac{\floor{\frac{m}{2}}}{\floor{\frac{m}{2}}-t+1}.$$
\item For $N \in \NN$ and $1<t<m$, the following are equivalent
\begin{enumerate} \item   for all  $s,r\ge 1$,
    \begin{align*}
        \frac{\alpha\left(\Pf_{2t}(Z)^{(s)}\right)}{s}\geq\frac{\alpha\left(\Pf_{2t}(Z)^{(r)}\right)+N-1}{r+N-1} \text{ (Demailly-type bound)}
    \end{align*}
    \item for all  $s\ge 1$,
    \begin{align*}
        \frac{\alpha\left(\Pf_{2t}(Z)^{(s)}\right)}{s}\geq\frac{\alpha\left(\Pf_{2t}(Z)\right)+N-1}{N} \text{ (Chudnovsky-type bound)}
    \end{align*}  \item   $N\geq \floor{\frac{m}{2}}-t+1$.\end{enumerate}
      \item For $N \in \NN$, 
    $$ \widehat{\rho}\left(\Pf_{2t}(Z)^{(\bullet)}, \left(\m^N\Pf_{2t}(Z)\right)^\bullet\right) = \frac{N+\alpha(\Pf_{2t}(Z))}{\widehat{\alpha}(\Pf_{2t}(Z))}=\frac{(N+t)(\floor{\frac{m}{2}}-t+1)}{\floor{\frac{m}{2}}}.$$
   Furthermore, if $(\min\{2t,m-2t\})!$ is invertible in $\k$, then $$ \rho\left(\Pf_{2t}(Z)^{(\bullet)}, \left(\m^N\Pf_{2t}(Z)\right)^\bullet\right)=  \frac{N+\alpha(\Pf_{2t}(Z))}{\widehat{\alpha}(\Pf_{2t}(Z))}=\frac{(N+t)(\floor{\frac{m}{2}}-t+1)}{\floor{\frac{m}{2}}}.$$
    \item For $N \in \NN$, 
    $$ \widehat{\rho}\left(\Pf_{2t}(Z)^{(\bullet)}, \overline{\m^N\left(\Pf_{2t}(Z)\right)^\bullet}\right) = \frac{\alpha(\Pf_{2t}(Z))}{\widehat{\alpha}(\Pf_{2t}(Z))}=\frac{t(\floor{\frac{m}{2}}-t+1)}{\floor{\frac{m}{2}}}.$$
    Furthermore, if $(\min\{2t,m-2t\})!$ is invertible in $\k$, then $$ \widehat{\rho}\left(\Pf_{2t}(Z)^{(\bullet)}, \m^N\left(\Pf_{2t}(Z)\right)^\bullet\right)=  \frac{\alpha(\Pf_{2t}(Z))}{\widehat{\alpha}(\Pf_{2t}(Z))}=\frac{t(\floor{\frac{m}{2}}-t+1)}{\floor{\frac{m}{2}}}.$$
    \end{enumerate}
\end{theorem}
\begin{proof}
The proof follows from \Cref{wald-family}, \Cref{chudnovsky and demailly bounds}, \Cref{asym-res} and \Cref{asym-res-HH} except the second part of $(3-4)$. When $(\min\{2t,m-2t\})!$ is invertible in $\k$, by \cite[Proposition 2.6]{Negri96} and \cite[Theorem 2.10]{IM16}, $\overline{\m^r\Pf_{2t}(Z)^s} =\m^r\Pf_{2t}(Z)^s$ for all $s,r$. Therefore, the second part of $(4)$ follows immediately. Using \cite[Corollary 4.7]{HKNN23}, we get $\rho\left(\Pf_{2t}(Z)^{(\bullet)}, \left(\m^N\Pf_{2t}(Z)\right)^\bullet\right)= \widehat{\rho}\left(\Pf_{2t}(Z)^{(\bullet)},  \left(\m^N\Pf_{2t}(Z)\right)^\bullet\right)$, i.e., the second part of $(3)$ follows.
\end{proof}
The remark analogous to \Cref{remakr on demaily applications} is true in this case too.

    \subsection{Ideal of minors of Hankel matrices:} 
    Let $x_1,\ldots,x_{n}$ be variables. For each $k \le n$, by $X_k$ we denote a $k \times (n+1-k)$ {\it Hankel matrix} whose $(i,j)$-th entry is $x_{i+j-1}$. Let $R= \k[x_1,\ldots,x_n]$ and $\mathfrak{m}=(x_1,\ldots,x_n)$. For $1 \le t \le \min\{k,n-k+1\}$, let $I_t(X_k)$ be the ideal of $t\times t$ minors  of $X_k$. We set $\a_t=I_t(X_k)$ for $ 1\le t \le \floor{\frac{n+1}{2}}$. In this case, \ref{P1} is satisfied as
    $$\mathcal{R}_s(I_t(X_k))= R[I_t(X_k)T,I_{t+1}(X_k)T^2,\dots, I_m(X_k)T^{m-t+1}]$$ (by \cite[Proposition 4.1]{C98}). Clearly, $\a_1=\mathfrak{m}$, $\alpha(\a_t) =t$ for all $ 1\le t \le \floor{\frac{n+1}{2}},$ and $m=\floor{\frac{n+1}{2}}. $ So, \ref{P2} is easily satisfied. Also, it follows from \cite[Theorem 3.12, Remark 4.6]{C98} that for all $ 1 \le t \le \floor{\frac{n+1}{2}}$, $N \in \NN$ and $s \ge 1$,
  \begin{align*}
        \overline{\m^{N}I_t(X_k)^s} =\m^{N}I_t(X_k)^s = I_1(X_k)^{N+ts} \bigcap \left(\bigcap_{j=2}^tI_t(X_k)^{(s(t-j+1))}\right).
    \end{align*}  
    Hence, $I_1(X_k),\ldots,I_{\floor{\frac{n+1}{2}}}(X_k)$ satisfy \ref{P3} and \ref{P4} for any non-negative integer $N$.

    \begin{theorem}
    Let $x_1,\ldots,x_{n}$ be variables. For each $j \le n$, let $X_j$ be the $j \times (n+1-j)$ Hankel matrix. Then, 
    \begin{enumerate}
        \item For $k \ge 1$ and $0\leq l\leq \floor{\frac{n+1}{2}}-t$, $
        \alpha\left(I_t(X_j)^{(k(\floor{\frac{n+1}{2}}-t+1)+l)}\right)=
                k\floor{\frac{n+1}{2}}-l.$
Furthermore, the Waldschmidt constant is given by the equation $$\widehat{\alpha}(I_t(X_j))=\frac{\floor{\frac{n+1}{2}}}{\floor{\frac{n+1}{2}}-t+1}.$$
\item  For $N \in \NN$ and $1<t<m$, the following are equivalent \begin{enumerate}
    \item    for all  $s,r\ge 1$,
    \begin{align*}
        \frac{\alpha\left(I_t(X_j)^{(s)}\right)}{s}\geq\frac{\alpha\left(I_t(X_j)^{(r)}\right)+N-1}{r+N-1} \text{ (Demailly-type bound)}
    \end{align*} 
    \item for all  $s\ge 1$,
    \begin{align*}
        \frac{\alpha\left(I_t(X_j)^{(s)}\right)}{s}\geq\frac{\alpha\left(I_t(X_j)\right)+N-1}{N} \text{ (Chudnovsky-type bound)}
    \end{align*} \item    $N\geq \floor{\frac{n+1}{2}}-t+1$. \end{enumerate}
    \item For $N \in \NN$,  
    $$ \rho(I_t(X_j)^{(\bullet)}, (\m^N I_t(X_j))^\bullet)= \widehat{\rho}(I_t(X_j)^{(\bullet)}, (\m^N I_t(X_j))^\bullet) = \frac{N+\alpha(I_t(X_j))}{\widehat{\alpha}(I_t(X_j))}.$$
    \item For $N \in \NN$,  
    $$  \widehat{\rho}(I_t(X_j)^{(\bullet)}, \m^N( I_t(X_j))^\bullet) = \frac{\alpha(I_t(X_j))}{\widehat{\alpha}(I_t(X_j))}.$$
    \end{enumerate}
\end{theorem}
\begin{proof}
The proof follows from \Cref{wald-family}, \Cref{chudnovsky and demailly bounds}, \Cref{asym-res-HH1}, and \Cref{asym-res-HH}. Since $\overline{\m^rI_t(X_j)^s} =\m^rI_t(X_j)^s$ for all $s,r$, by \cite[Corollary 4.7]{HKNN23}, $\rho(I_t(X_j)^{(\bullet)}, (\m^N I_t(X_j))^\bullet)= \widehat{\rho}(I_t(X_j)^{(\bullet)}, (\m^N I_t(X_j))^\bullet).$ Hence, the assertion follows. 
\end{proof}
As mentioned in the introduction, the theory built-in \Cref{sec 2,Sec3} is not restricted to classical varieties, although these provide the first place to test the hypothesis. While  \Cref{self linked example,star config example} satisfy the framework, it remains an open question whether a broader class of examples can be constructed within the setting of \Cref{sec 2,Sec3}. 

\bibliographystyle{siam}
\bibliography{references}

\end{document}